\documentclass{amsart}
\usepackage[latin1]{inputenc}
\usepackage{epsfig}
\usepackage{color}
\usepackage[british,english]{babel}
\usepackage{amsthm}
\usepackage{amsmath}
\usepackage{amsfonts}
\usepackage{amssymb}

\usepackage{graphicx}

\newtheorem{corollary}{Corollary}[section]

\newtheorem{lemma}[corollary]{Lemma}
\newtheorem{prp}[corollary]{Proposition}
\newtheorem{remark}[corollary]{Remark}
\newtheorem{thm}[corollary]{Theorem}

\newfont{\sBlackboard}{msbm10 scaled 900}

\newcommand{\mylabel}[1]{\label{#1}
            \ifx\undefined\stillediting
            \else \fbox{$#1$}\fi }
\newcommand{\BE}{\begin{equation}}

\newcommand{\EEQ}{\end{equation}}
\newcommand{\rfb}[1]{\mbox{\rm
   (\ref{#1})}\ifx\undefined\stillediting\else:\fbox{$#1$}\fi}

\newfont{\Blackboard}{msbm10 scaled 1200}

\newfont{\roma}{cmr10 scaled 1200}
\def\RR{{\mathbb R} }

\newcommand{\mm}    {{\hbox{\hskip 0.5pt}}}

\newcommand{\bluff} {{\hbox{\raise 15pt \hbox{\mm}}}}
scaled\magstep2

\makeatletter
\def\section{\@startsection {section}{1}{\z@}{-3.5ex plus -1ex minus
    -.2ex}{2.3ex plus .2ex}{\large\bf}}
\makeatother

\def\be{\begin{equation}}
\def\ee{\end{equation}}

\begin{document}
\title[Caffarelli-Kohn-Nirenberg inequality]{A weighted anisotropic variant of the Caffarelli-Kohn-Nirenberg inequality and applications }

\author[A. Bahrouni]{Anouar Bahrouni}
\address[A. Bahrouni]{Mathematics Department, University of Monastir,
Faculty of Sciences, 5019 Monastir, Tunisia}
\email{\tt bahrounianouar@yahoo.fr}

\author[V.D. R\u{a}dulescu]{Vicen\c{t}iu D. R\u{a}dulescu}
\address[V.D. R\u{a}dulescu]{Faculty of Applied Mathematics, AGH University of Science and Technology, al. Mickiewicza 30, 30-059 Krak\'ow, Poland \& Institute of Mathematics ``Simion Stoilow" of the Romanian Academy, P.O. Box 1-764,
          014700 Bucharest, Romania}
\email{\tt vicentiu.radulescu@imar.ro}

\author[D.D. Repov\v{s}]{Du\v{s}an D. Repov\v{s}}
\address[D.D. Repov\v{s}]{Faculty of Education and Faculty of Mathematics and Physics, University of Ljubljana,
 SI-1000 Ljubljana, Slovenia}
\email{\tt dusan.repovs@guest.arnes.si}


\begin{abstract}
 We present a weighted version of the Caffarelli-Kohn-Nirenberg inequality in the framework of variable exponents. The combination of this inequality with a variant of the fountain theorem, yields the existence of infinitely many solutions for a class of non-homogeneous
problems with Dirichlet boundary condition.
\end{abstract}

\maketitle

\section{Introduction}
Nonlinear problems with variable exponents are motivated by numerous models in the applied sciences that are driven by some classes of non-homogeneous partial differential operators. In some circumstances, the standard analysis based on the theory of usual Lebesgue and Sobolev function  spaces,  $L^p$ and $W^{1,p}$, is not appropriate in the framework of materials that involve non-homogeneities. For instance, both electro-rheological ``smart" fluids  and phenomena arising in image processing are properly described by nonlinear models in which the exponent $p$ is not necessarily constant.
The variable exponent describes the geometry of the material which is allowed to change its hardening  exponent at different points.
This leads to the analysis of variable exponent Lebesgue and Sobolev function spaces (denoted by $L^{p(x)}$ and $W^{1,p(x)}$), where $p$ is a real-valued (non-constant) function. We point out important contributions of Halsey \cite{halsey} and Zhikov \cite{zh2} in strong relationship with the  behavior of strongly anisotropic materials. This is mainly achieved
in the framework of the homogenization and nonlinear elasticity.
We refer, e.g., to Acerbi and Mingione \cite{acerbi} and Ru\v{z}i\v{c}ka
\cite{ruz} (electrorheological ``smart" fluids) and Antontsev and
Shmarev \cite{shm} (nonlinear Darcy's law in porous media).
  A thorough variational analysis of the problems with variable exponents has been developed in the recent monograph by R\u adulescu and Repov\v s \cite{radrep} (see also the survey paper by R\u adulescu \cite{radnla} and the important contributions of Pucci {\it et al.} \cite{patrizia1, patrizia2}).

Let $\Omega\subset\RR^N$ ($N\geq 2$) be a bounded domain with smooth boundary.  The following Caffarelli-Kohn-Nirenberg inequality \cite{caf} states that
given $p\in(1,N)$ and real numbers $a,\,b$ and $q$ such that
$$-\infty<a<\frac{N-p}{p},\;\;\; a\leq b\leq a+1,\;\;\; q=\frac{Np}{N-p(1+a-b)}\,,$$
 there is a
positive constant $C_{a,b}$ such that
 for all $u\in C_c^1(\Omega)$,
\begin{equation}\label{CKNineq}\left(\int_{\Omega}|x|^{-bq}|u|^q\;dx\right)^{p/q}\leq C_{a,b}\int_{\Omega}|x|^{-ap}|\nabla u|^p\;dx\,.\end{equation}
This result goes back to the celebrated Hardy inequality \cite{hardy}, which establishes that if $1\leq p<N$, then
for all $u\in C^\infty_0(\RR^N\setminus\{0\})$
$$\left\|\frac{u(x)}{\|x\|} \right\|_{L^p(\RR^N)}\leq \frac{p}{N-p}\,\|\nabla u\|_{L^p(\RR^N)},$$
where $\|x\|=\sqrt{x_1^2+\cdots +x_n^2}$ and the constant $\frac{p}{N-p}$ is known to be sharp.
Inequality \eqref{CKNineq} has been widely analyzed in many different settings (see, e.g., \cite{abdellaoui, abdellaui, ad, wang, chern, ambr, ambro, gr, gaz}). Nowadays, there is vast literature on this subject,
for example, the MathSciNet search shows about 5000 research works related to this topic.

The main aim of this paper is to present an analogoue of the Caffarelli-Kohn-Nirenberg inequality in the framework of variable exponents. To the best of our knowledge, there are very few results dealing with this topic.
 For instance,
the following result was established in \cite{mh}: there exists a positive constant $C$ such that
\begin{equation}\label{cafrad}
\int_\Omega|u(x)|^{p(x)}\;dx\leq C\int_\Omega|\overrightarrow
a(x)|^{p(x)}|\nabla u(x)|^{p(x)}\;dx,\;\;\;\mbox{for all}\;u\in
C_c^1(\Omega),
\end{equation}
where  $\Omega\subset\RR^N$ ($N\geq 2$) is a bounded
domain with smooth boundary, while $\overrightarrow a: \Omega\rightarrow
\mathbb{R}^{N}$ and $p: \Omega\rightarrow \mathbb{R}$ are functions
of class $C^{1}$, satisfying for some $a_0>0$
\begin{equation}\label{rada}
{\rm div}\,(\overrightarrow a(x))\geq a_{0}>0, \ \ \mbox{for all}\ x\in \Omega,
\end{equation}
provided that
\begin{equation}\label{radp}
\overrightarrow a(x)\cdot\nabla p(x)=0 ,\ \ \mbox{for all}\ x\in \Omega.
\end{equation}

In this paper, we establish a more involved version of inequality \eqref{cafrad}, which combines the contributions of several quantities.
 In order to introduce the main abstract result of the paper, we assume that  $\Omega\subset\RR^N$ ($N\geq 2$) is a bounded
domain with smooth boundary and
$a,p:\overline{\Omega}\rightarrow\RR$ are given functions such that
the following hypotheses are fulfilled:

(A) $a$ is a function of class $C^{1}$ and there exist $x_{0}\in
\Omega$, $r>0$,  $s \in (1,+\infty)$ such that
$$|a(x)|\neq 0, \ \ \mbox{for all}\ x\in \overline{\Omega}\setminus \{x_{0}\} \ \ \mbox{and} \ \  |a(x)|\geq |x-x_{0}|^{s}, \ \ \mbox{for all}\ x\in B(x_{0},r);\ \mbox{and}$$

(P) $p$ is a function of class $C^{1}$ and $p(x)\in (2,N)$ for all
$x\in\Omega$.

The main abstract result of this paper is the following
weighted Caffarelli-Kohn-Nirenberg inequality.

\begin{thm}\label{caf}
Assume that conditions $(A)$ and $(P)$ hold. Let 
$\Omega\subset\RR^N$ ($N\geq 2$) be a bounded domain with smooth
boundary. Then there exists a positive constant $\beta$ such that
\begin{align*}
\displaystyle \int_{\Omega}|a(x)|^{p(x)}|u(x)|^{p(x)}dx &\leq \beta
 \int_{\Omega}|a(x)|^{p(x)-1}||\nabla a(x)||u(x)|^{p(x)}dx\\
&+\beta\left(\displaystyle \int_{\Omega}|a(x)|^{p(x)}|\nabla
u(x)|^{p(x)}dx+\displaystyle \int_{\Omega}|a(x)|^{p(x)}|\nabla p(x)|
|u(x)|^{p(x)+1}dx\right)\\
&+\beta \displaystyle \int_{\Omega}|a(x)|^{p(x)-1}|\nabla p(x)|
|u(x)|^{p(x)-1}dx.
\end{align*}
 for all $u\in C_{c}^{1}(\Omega)$.
\end{thm}

 We point out that, by hypotheses $(A)$ and $(P)$, the potential $\nabla
a$ can vanish in $\Omega$ and we do not assume that $\nabla p(x)\cdot
a(x)=0$, for all $x\in \Omega$ (see assumption \eqref{radp} related to inequality \eqref{rada}).

Next, we are concerned with the existence of infinitely many
solutions for the problem
\begin{equation}\label{imain}
  \left\{\begin{array}{ll} &\displaystyle -{\rm div}\,(B(x)|\nabla
u|^{p(x)-2}\nabla
u)+(A(x)|u|^{p(x)-2}+C(x)|u|^{p(x)-3})u=\\
&\displaystyle (b(x)|u|^{q(x)-2}-D(x)|u|^{p(x)-1})u
\quad\mbox{in
}\phantom{\partial}\Omega,\\
&u=0 \quad\mbox{on}\ \partial\Omega\,,
\end{array}\right.
\end{equation}
where the variable exponent $q$ fulfills a subcritical condition (namely, in the sense of Sobolev-type embeddings for spaces with variable exponent).
 We assume that $b:\Omega\rightarrow\RR$ and  the weighted potentials $A$, $B$, $C$, $D$ are defined by
 \begin{equation}\label{ABCD}
 \left\{\begin{array}{ll}
 &\displaystyle A(x)=|a(x)|^{p(x)-1}|\nabla a(x)|\\
 &\displaystyle B(x)=|a(x)|^{p(x)}\\
 &\displaystyle C(x)=|a(x)|^{p(x)-1}|\nabla p(x)|\\
 &\displaystyle D(x)=B(x)|\nabla p(x)|.
 \end{array}\right.
 \end{equation}
 

The potential $b$ is assumed to satisfy the following hypothesis:

\smallskip
(B) $b\in L^{\infty}(\Omega)$ and $b>0$ in $\Omega$.

\smallskip
In the final part of this paper, by combining our generalized Caffarelli-Kohn-Nirenberg inequality with
a variant of the fountain theorem, we shall
prove that problem \eqref{imain} has infinitely many
solutions.

\section{Terminology and the abstract setting}
In this section we recall some basic definitions and properties
concerning the Lebesgue and Sobolev spaces with variable exponent. We
refer to \cite{Fan, radrep} and the references
therein.

Consider the set
$$C_+(\overline\Omega)=\{p\in C(\overline\Omega);\;p(x)>1\;{\rm
for}\; {\rm all}\;x\in\overline\Omega\}.$$ For all $p\in
C_+(\overline\Omega)$ we define
    $$p^+=\sup_{x\in\Omega}p(x)\qquad\mbox{and}\qquad p^-=
    \inf_{x\in\Omega}p(x).$$ For any $p\in C_+(\overline\Omega)$, we
    define the {\it variable exponent Lebesgue space}
    $$L^{p(x)}(\Omega)=\left\{u;\ u\ \mbox{is
measurable real-valued function such that }
 \int_\Omega|u(x)|^{p(x)}\;dx<\infty\right\}.$$ This vector space is
a Banach space if it is endowed with the {\it Luxemburg norm}, which
 is defined by
 $$|u|_{p(x)}=\inf\left\{\mu>0;\;\int_\Omega\left|
 \frac{u(x)}{\mu}\right|^{p(x)}\;dx\leq 1\right\}.$$ The function
 space $L^{p(x)}(\Omega)$ is reflexive if and only if $1 < p^-\leq
 p^+<\infty$
  and continuous functions with compact support
 are dense in $L^{p(x)}(\Omega)$ if $p^{+}<\infty$.

 Let $L^{q(x)}(\Omega)$ denote the conjugate space of
 $L^{p(x)}(\Omega)$, where $1/p(x)+1/q(x)=1$. If $u\in
 L^{p(x)}(\Omega)$ and $v\in L^{q(x)}(\Omega)$ then  the following
 H\"older-type inequality holds:
 \begin{equation}\label{Hol}
 \left|\int_\Omega uv\;dx\right|\leq\left(\frac{1}{p^-}+
 \frac{1}{q^-}\right)|u|_{p(x)}|v|_{q(x)}\,.
 \end{equation}
 Moreover, if $p_j\in C_+(\overline\Omega)$ ($j=1,2,3$) and
 $$\frac{1}{p_1(x)}+\frac{1}{p_2(x)}+\frac{1}{p_3(x)}=1$$
 then for all $u\in L^{p_1(x)}(\Omega)$, $v\in L^{p_2(x)}(\Omega)$,
 $w\in L^{p_3(x)}(\Omega)$
 \begin{equation}\label{Hol1}
 \left|\int_\Omega uvw\;dx\right|\leq\left(\frac{1}{p_1^-}+
 \frac{1}{p_2^-}+\frac{1}{p_3^-}\right)|u|_{p_1(x)}|v|_{p_2(x)}|w|_{p_3(x)}\,.
 \end{equation}

 The inclusion between Lebesgue spaces also generalizes the classical
 framework, namely if $0 <|\Omega|<\infty$ and $p_1$, $p_2$ are
 variable exponents such that $p_1\leq p_2$
  in $\Omega$ then there exists a continuous embedding
 $L^{p_2(x)}(\Omega)\hookrightarrow L^{p_1(x)}(\Omega)$.
 \begin{prp}\label{pr1}
 If we denote
 $$\rho(u)=\displaystyle \int_{\Omega}|u|^{p(x)}dx, \ \ \forall u\in L^{p(x)}(\Omega),$$
 then\\
 $(i)$ $|u|_{p(x)}<1 (=1;>1)\Leftrightarrow \rho(u)<1(=1;>1)$;\\
 $(ii)$ $|u|_{p(x)}>1 \Rightarrow |u|_{p(x)}^{p^{-}}\leq \rho(u) \leq
 |u|_{p(x)}^{p^{+}}$;\\
 $(iii)$ $|u|_{p(x)}<1 \Rightarrow |u|_{p(x)}^{p^{+}}\leq \rho(u)
 \leq |u|_{p(x)}^{p^{-}}$.
 \end{prp}
 \begin{prp}\label{pr2}
 If $u,u_{n}\in L^{p(x)}(\Omega)$ and  $n\in \mathbb{N}$, then the
 following statements are equivalent:\\
 $(1)$ $\displaystyle \lim_{n\rightarrow +\infty}
 |u_{n}-u|_{p(x)}=0$;\\
 $(2)$  $\displaystyle \lim_{n\rightarrow +\infty} \rho(
 u_{n}-u)=0;$\\
 $(3)$ $u_{n}\rightarrow u$ in measure on $\Omega$ and
 $\displaystyle \lim_{n\rightarrow +\infty}\rho(u_{n})=\rho(u).$
 \end{prp}

  If $k$ is a positive integer and $p\in C_+(\overline\Omega)$, then we define the variable exponent  Sobolev space by
  $$
  W^{k,p(x)}(\Omega)=\{u\in L^{p(x)}(\Omega);\; D^\alpha u\in
  L^{p(x)} (\Omega),\ \mbox{for all}\ |\alpha|\leq k \}.
  $$
  Here $\alpha =(\alpha_1,\ldots ,\alpha_N)$ is a multi-index, $|\alpha|=\sum_{i=1}^N\alpha_i$, and
  $$D^\alpha u=\frac{\partial^{|\alpha|}u}{\partial^{\alpha_1}_{x_1}\ldots \partial^{\alpha_N}_{x_N}}\,.$$

  On $W^{k,p(x)}(\Omega)$ we  consider the following norm
  $$
  \|u\|_{k,p(x)}=\sum_{|\alpha|\leq k}|D^\alpha u|_{p(x)}.
  $$
  Then $W^{k,p(x)}(\Omega)$ is a reflexive and separable Banach space if $1<p^{-}\leq p^{+}<+\infty$. Let $W^{k,p(x)}_0(\Omega)$
  denote the closure of $C^\infty_0(\Omega)$ in $W^{k,p(x)}(\Omega)$.

The  Lebesgue and Sobolev
spaces with variable exponents coincide with the usual Lebesgue and Sobolev spaces provided that $p$ is constant.
 According to R\u adulescu and Repov\v{s} \cite[pp. 8-9]{radrep}, these function spaces have some unusual properties,
such as:

(i)
Assuming that $1<p^-\leq p^+<\infty$ and $p:\overline\Omega\rightarrow [1,\infty)$ is a smooth function,  the following co-area formula
$$\int_\Omega |u(x)|^pdx=p\int_0^\infty t^{p-1}\,|\{x\in\Omega ;\ |u(x)|>t\}|\,dt$$
has  no analogue in the framework of variable exponents.

(ii) Spaces $L^{p(x)}$ do {\it not} satisfy the {\it mean continuity property}. More exactly, if $p$ is nonconstant and continuous in an open ball  $B(x_0)$,
then there is some $u\in L^{p(x)}(B(x_0))$ such that  $u(x+h)\not\in L^{p(x)}(B(x_0+h))$ for every $h\in\RR^N$ with arbitrary small norm.

(iii) Function spaces with variable exponent
 are {\it never} invariant with respect to translations.  The convolution is also limited. For instance,  the classical Young inequality
$$| f*g|_{p(x)}\leq C\, | f|_{p(x)}\, \| g\|_{L^1}$$
remains valid if and only if
$p$ is constant.

\section{Weighted Caffarelli-Kohn-Nirenberg inequality for $p(x)$-Laplacian}

We start with the following weighted logarithmic inequality.
\begin{lemma}\label{log}
Let condition $(P)$ be satisfied. Then there exists a positive
constant $\mu$ such that $$
\int_{{\rm supp}\,(u)}|\nabla p(x)||u(x)|^{p(x)}|\log(|u(x)|)|dx \leq \mu
\displaystyle \int_{\Omega}|\nabla
p(x)|\left(|u(x)|^{p(x)-1}+|u(x)|^{p(x)+1}\right)dx,$$
 for all $u\in C_{c}^{1}(\Omega).$
\end{lemma}
\begin{proof}
Let $u\in C_{c}^{1}(\Omega)$. We define
$$\alpha_{1}=\displaystyle \sup_{0<t\leq 1} t|\log(t)|<\infty \ \ \mbox{and} \ \ \alpha_{2}=\displaystyle \sup_{1<t} t^{-1}\log(t)<\infty.$$
We observe that $0<\alpha_{1}<+\infty$ and
$0<\alpha_{2}<+\infty.$ Let
$$\Omega_{1}=\{x\in {\rm supp}\,(u); \ \ |u(x)|\leq 1\} \ \ \mbox{and} \ \ \Omega_{2}=\{x\in {\rm supp}\, (u); \ \ |u(x)|> 1\}.$$
Then
$$\begin{array}{ll}
\displaystyle \int_{{\rm supp}\,(u)}|\nabla
p(x)||u(x)|^{p(x)}|\log(|u(x)|)|dx &= \displaystyle
\int_{\Omega_{1}}|\nabla p(x)||u(x)|^{p(x)}|\log(|u(x)|)|dx \\
&+\displaystyle
\int_{\Omega_{2}}|\nabla p(x)||u(x)|^{p(x)}|\log(|u(x)|)|dx \\
&\leq \alpha_{1} \displaystyle \int_{\Omega_{1}}|\nabla
p(x)||u(x)|^{p(x)-1}dx\nonumber\\&+\alpha_{2}\displaystyle
\int_{\Omega_{2}}|\nabla p(x)||u(x)|^{p(x)+1}dx\\
&\leq \mu(\displaystyle \int_{\Omega}|\nabla p(x)||u(x)|^{p(x)-1}dx\\
&+ \displaystyle \int_{\Omega}|\nabla p(x)||u(x)|^{p(x)+1}dx),
\end{array}
$$
 where $\mu=\max(\alpha_{1},\alpha_{2})$. This proves the lemma.
\end{proof}

\textbf{Proof of Theorem} \ref{caf}.  We prove in what
follows our weighted version of the Caffarelli-Kohn-Nirenberg inequality
with variable exponent.

 We define the function $W:
\mathbb{R}^{N} \rightarrow \mathbb{R}^{N}$ by $W(y)=y$ for all $y\in
\mathbb{R}^{N}.$ We choose $\epsilon>0$ small enough so that
\begin{equation}\label{epslon}
0<\epsilon <\frac{N}{p^{+}\|W\|_{L^{\infty}(\Omega)}}.
\end{equation}
 By a straightforward
computation we can deduce that for all $u\in C_{c}^{1}(\Omega)$ we have
\begin{align}\label{comp}
{\rm div}\,(|a(x)u(x)|^{p(x)} W(x))&=|a(x)|^{p(x)} |u(x)|^{p(x)}
{\rm div}\,(W(x))\\
&+\displaystyle p(x)|a(x)|^{p(x)}|u(x)|^{p(x)-2} u(x) \nabla
u(x)\cdot W(x)\nonumber\\
&+p(x)|u(x)|^{p(x)} |a(x)|^{p(x)-2}a(x) \nabla a(x)\cdot W(x)\nonumber \\
&+|u(x)a(x)|^{p(x)} \log(|a(x)u(x)|) \nabla p(x).W(x), \ \ \forall
x\in \Omega.
\end{align}
Now the flux-divergence theorem implies that $\displaystyle
\int_{\Omega}{\rm div}\,(|a(x)u(x)|^{p(x)} W(x))dx=0$. It follows from Lemma
\ref{log} and conditions $(A)$ and $(P)$, that
\begin{align}\label{21}
\displaystyle \int_{\Omega}|a(x)u(x)|^{p(x)} {\rm div}\,(W(x))dx &\leq
p^{+}\displaystyle \int_{\Omega}|u(x)|^{p(x)}|a(x)|^{p(x)-1}|\nabla
a(x)|
|W(x)|dx \nonumber\\
&+\displaystyle
\int_{\Omega}|a(x)u(x)|^{p(x)}|\log(|u(x)a(x)|)||\nabla
p(x)||W(x)|dx \nonumber\\
&+p^{+}\displaystyle
\int_{\Omega}|u(x)|^{p(x)-1}|a(x)|^{p(x)}|\nabla
u(x)| |W(x)|dx \nonumber \\
&\leq p^{+}\|W\|_{L^{\infty}(\Omega)} \displaystyle
\int_{\Omega}|a(x)|^{p(x)-1}||\nabla a(x)| |u(x)|^{p(x)}dx \nonumber\\
&+\mu \|W\|_{L^{\infty}(\Omega)} \displaystyle
\int_{\Omega}|a(x)|^{p(x)-1}|\nabla p(x)||u(x)|^{p(x)-1}dx
\nonumber\\
&+\mu \|a\|_{L^{\infty}(\Omega)}\|W\|_{L^{\infty}(\Omega)}
\displaystyle \int_{\Omega}|a(x)|^{p(x)}|\nabla p(x)||u(x)|^{p(x)+1}dx \nonumber\\
&+  \epsilon p^{+}\|W\|_{L^{\infty}(\Omega)} \displaystyle
\int_{\Omega}|a(x)|^{p(x)}|u(x)|^{p(x)}dx \nonumber\\ &+
p^{+}\frac{\|W\|_{L^{\infty}(\Omega)}}{\epsilon^{p^{-}-1}}
\displaystyle \int_{\Omega}|a(x)|^{p(x)}|\nabla u(x)|^{p(x)}dx.
\end{align}
 Next, we combine ${\rm div}\,(W(x))=N$ in $\Omega$ with relation \eqref{21}
and the following Young inequality:
$$a^{p-1}b\leq\varepsilon a^p +\frac{b^p}{\varepsilon^{p-1}}\,,\quad\mbox{for all}\ a,b,\varepsilon\in(0,\infty),\ p\in (1,\infty).$$
It follows that

\begin{align}\label{22}
(N-p^{+}\|W\|_{L^{\infty}(\Omega)} \epsilon)\displaystyle
\int_{\Omega}|a(x)u(x)|^{p(x)}dx &\leq
p^{+}\frac{\|W\|_{L^{\infty}(\Omega)}}{\epsilon
^{p^{-}-1}}\displaystyle
\int_{\Omega}|a(x)|^{p(x)}|\nabla u(x)|^{p(x)}dx  \nonumber\\
&+p^{+}\|W\|_{L^{\infty}(\Omega)}\displaystyle
\int_{\Omega}|a(x)|^{p(x)-1}||\nabla
a(x)||u(x)|^{p(x)}dx \nonumber\\
&+ c\displaystyle \int_{\Omega}|a(x)|^{p(x)}|\nabla
p(x)||u(x)|^{p(x)+1}dx\\
&+\mu\|W\|_{L^{\infty}(\Omega)}\displaystyle
\int_{\Omega}|a(x)|^{p(x)-1}|\nabla p(x)||u(x)|^{p(x)-1}dx,
\end{align}
with $c=\mu \|a\|_{L^{\infty}(\Omega)}\|W\|_{L^{\infty}(\Omega)}$.
Invoking \eqref{epslon}, we  set
$$\beta=\frac{\max (c,p^{+}\frac{\|W\|_{L^{\infty}(\Omega)}}{\epsilon
^{p^{-}-1}},p^{+}
\|W\|_{L^{\infty}(\Omega)},\mu\|W\|_{L^{\infty}(\Omega)})}{(N-p^{+}\|W\|_{L^{\infty}(\Omega)}
\epsilon)}.$$

This completes the proof of Theorem \ref{caf}.\qed

\medskip
We denote by $W^{1,p(x)}_{0,a(x)}(\Omega)$ the closure of
$C_{c}^{1}(\Omega)$ under the norm
$$\begin{array}{ll}
\|u\|=&\displaystyle ||B(x)|^{\frac{1}{p(x)}}\nabla u(x)|_{p(x)}+|A(x)^{\frac{1}{p(x)}}u(x)|_{p(x)}+||D(x)|^{\frac{1}{p(x)+1}}u(x)|_{p(x)+1}+\\
&\displaystyle ||C(x)|^{\frac{1}{p(x)-1}}u(x)|_{p(x)-1},\end{array}$$
where the potentials $A$, $B$, $C$, $D$ are defined in \eqref{ABCD}. \\
(There is no modification since the norm on $L^{p(x)}
$ is denoted by $|\ |_{p(x)}$. But in the above equality the first and the fourth bars are
for the norm on $L^{p(x)}$, while the second and the third bars denote the absolute value of
$B$, $A$, $D$, and $C$.)

 As a corollary of Theorem \ref{caf},
we prove the following compactness property.

\begin{lemma}\label{com}
Assume that conditions $(A)$ and $(P)$ hold. Furthermore, assume
that $p^{-}>1+s$. Then $W^{1,p(x)}_{0,a(x)}(\Omega)$ is compactly
embedded in $L^{q}(\Omega)$ for each $q\in
(1,\frac{Np^{-}}{N+sp^{+}})$. Moreover, the same conclusion holds if we  replace $L^{q}(\Omega)$ by $L^{q(x)}(\Omega)$, provided that $q^+<
\frac{Np^{-}}{N+sp^{+}}$.
\end{lemma}
\begin{proof}
Fix $q\in
(1,\frac{Np^{-}}{N+sp^{+}})$. Let $(u_{n})$ be a bounded sequence in
$W^{1,p(x)}_{0,a(x)}(\Omega)$. Since $x_{0}\in \Omega$, it follows
that there exists $\epsilon_{0}>0$ such that
$$0<\epsilon_{0}<\min(1,r) \; \mbox{and} \; \overline{B}(x_{0}, \epsilon_{0})\subset \Omega.$$
 Fix $\epsilon>0$ so that $\epsilon
<\epsilon_{0}$.  From condition $(A)$, there exists
$a_{0}>0$ such that $a(x)\geq a_{0}$, for all $x\in \Omega \setminus
\overline{B}(x_{0}, \epsilon) $. Hence, by invoking Theorem \ref{caf} we
deduce that the sequence $(u_{n})$ is bounded in $L^{p(x)}(\Omega
\setminus \overline{B}(x_{0}, \epsilon))$. Consequently, $(u_{n})$
is  bounded in $W^{1,p(x)}(\Omega\setminus \overline{B}(x_{0},
\epsilon) )$. Since $W^{1,p(x)}(\Omega\setminus \overline{B}(x_{0},
\epsilon) )\subset W^{1,p^{-}}(\Omega\setminus \overline{B}(x_{0},
\epsilon)) $ we deduce that $(u_{n})$ is bounded
 in $W^{1,p^{-}}(\Omega\setminus \overline{B}(x_{0},
\epsilon))$.
For all $s>0$ we have $Np^-/(N-p^-)>Np^-/(N+sp^+)$. Thus, since $1<q<Np^-/(N+sp^+)$, 
the classical compact embedding theorem shows that
there exists a convergent subsequence of $(u_{n})$, still denoted by
$(u_{n})$, in $L^{q}(\Omega\setminus \overline{B}(x_{0},
\epsilon))$. Thus, for any large enough $n$ and $m$ we have
\begin{equation}\label{31}
\displaystyle \int_{\Omega\setminus \overline{B}(x_{0},
\epsilon)}|u_{n}-u_{m}|^{q}dx <\epsilon.
\end{equation}
Now the H\"older inequality for variable exponent spaces implies
\begin{align}\label{32}
\displaystyle \int_{ B(x_{0}, \epsilon)}|u_{n}-u_{m}|^{q}dx
&=\displaystyle \int_{ B(x_{0},
\epsilon)}|a(x)|^{q}|a(x)|^{-q}|u_{n}-u_{m}|^{q}dx \nonumber\\
&\leq c \||a(x)|^{-q} \chi_{B(x_{0},
\epsilon)}\|_{(\frac{p(x)}{q})'} \||a(x)|^{q}
|u_{n}-u_{m}|^{q}\|_{\frac{p(x)}{q}},
\end{align}
where $c$ is a positive constant and
$(\frac{p(x)}{q})'=\frac{p(x)}{p(x)-q}$.
By Theorem \ref{caf} and Proposition \ref{pr1}, there exist
 positive constants $c_{1}$ and  $c_{2}$  such that
\begin{align}\label{33}
\||a(x)|^{q} |u_{n}-u_{m}|^{q}\|_{\frac{p(x)}{q}}&\leq
c_{1}(\displaystyle
\int_{\Omega}|a(x)|^{p(x)}|u_{n}-u_{m}|^{p(x)}dx)^{\frac{q}{p^{-}}}
\nonumber\\&+c_{1}(\displaystyle
\int_{\Omega}|a(x)|^{p(x)}|u_{n}-u_{m}|^{p(x)}dx)^{\frac{q}{p^{+}}}
\nonumber \\ &\leq c_{2}.
\end{align}
Taking into account relations \eqref{32} and \eqref{33} we deduce
that
\begin{equation}\label{34}
\displaystyle \int_{ B(x_{0}, \epsilon)}|u_{n}-u_{m}|^{q}dx \leq
c_{2} \||a(x)|^{-q} \chi_{B(x_{0},
\epsilon)}\|_{(\frac{p(x)}{q})'}.
\end{equation}
By invoking Proposition \ref{pr1}, we obtain
\begin{align}\label{35}
\||a(x)|^{-q} \chi_{B(x_{0},
\epsilon)}\|_{(\frac{p(x)}{q})'}&\leq (\displaystyle
\int_{\Omega}|a(x)|^{\frac{-qp(x)}{(p(x)-q)}}\chi_{B(x_{0},
\epsilon)}dx)^{((\frac{p(x)}{q})')^{+}} \nonumber \\
&+(\displaystyle \int_{\Omega}|a(x)|^{\frac{-qp(x)}{(p(x)-q)}}
\chi_{B(x_{0}, \epsilon)}dx)^{((\frac{p(x)}{q})')^{-}}.
\end{align}
Using condition $(A)$ and  $\epsilon <1$, we infer that
 \begin{align}\label{36}
\displaystyle
\int_{B(x_{0},\epsilon)}|a(x)|^{\frac{-qp(x)}{(p(x)-q)}}dx&\leq
\displaystyle \int_{B(0,\epsilon)}|x|^{\frac{-sqp^{+}}{(p^{-}-q)}}dx \nonumber\\
&= w_{n}\displaystyle
\int_{0}^{\epsilon}r^{N-1}r^{\frac{-sqp^{+}}{(p^{-}-q)}}dr\\
&= w_{n}\frac{\epsilon^{\alpha}}{\alpha},
\end{align}
where $\alpha=N-\frac{sqp^{+}}{(p^{-}-q)}>0$ and $w_{N}$ is the area
of the unit ball in $\mathbb{R}^{N}$. Thus, it follows from \eqref{31},
\eqref{34} and \eqref{36}  that
$$\displaystyle \int_{\Omega}|u_{n}-u_{m}|^{q}dx\leq c (\epsilon+\epsilon^{\alpha_{1}}+\epsilon^{\alpha_{2}}),$$
where $c$ is a positive constant,
$\alpha_{1}=((\frac{p(x)}{q})')^{-}\alpha$, and
$\alpha_{2}=((\frac{p(x)}{q})')^{+}\alpha$. We conclude that
$(u_{n})$ is a Cauchy sequence in $L^{q}(\Omega)$.

The same proof
still applies if we replace $L^{q}(\Omega)$ by $L^{q(x)}(\Omega)$. The conclusion of the lemma is now evident.
\end{proof}

\section{A multiplicity property for a problem with variable exponent}
In this section, we  work under conditions introduced in Lemma
\ref{com}. We investigate the existence of infinitely many solutions
of problem \eqref{imain}, where  $b \in L^{\infty}(\Omega)$ and
\begin{equation}\label{crucial} q(x) \in
\left(1,\min\left\{\frac{Np^{-}}{N+sp^{+}},p^{-}-1\right\}\right)\quad \mbox{for all $x\in \Omega$.}
\end{equation}

 We say that $u\in W^{1,p(x)}_{0,a(x)}(\Omega)$ is a {\it weak
solution} of problem \eqref{imain} if
\begin{align*}
\displaystyle \int_{\Omega} B(x)& |\nabla u(x)|^{p(x)-2} \nabla u(x)
\nabla v(x)dx + \displaystyle \int_{\Omega} A(x)|u(x)|^{p(x)-2}u(x)
v(x)dx \\ &+ \displaystyle \int_{\Omega} D(x)|u(x)|^{p(x)-1}u(x)
v(x)dx+\displaystyle \int_{\Omega} C(x)|u(x)|^{p(x)-3}u(x)v(x)dx\\&-
\int_{\Omega}b(x) |u(x)|^{q(x)-2}u(x) v(x)dx=0,
\end{align*}
for all $v\in  W^{1,p(x)}_{0,a(x)}(\Omega)$.

Standard argument can
be used  to show that $(W^{1,p(x)}_{0,a(x)}(\Omega),\| .
\|)$ is a reflexive Banach separable space. Then, by \cite{china},
there exist $\left(e_{n}\right)\subset W^{1,p(x)}_{0,a(x)}(\Omega)$
and $e_{n}^{\ast}\subset (W^{1,p(x)}_{0,a(x)}(\Omega))^{\ast}$ such
that
$$e_{n}^{\ast}\left(e_{m}\right)=1 \ \ \mbox{ if} \ \  n=m \ \ \mbox{ and} \ \
e_{n}^{\ast}\left(e_{m}\right)=0 \ \ \mbox{if} \ \ n\neq m.$$ It
follows that
$$W^{1,p(x)}_{0,a(x)}(\Omega)=\overline{\mbox{span}}\left\{e_{n}, \  n\geq 1\right\}
\ \ \mbox{and} \ \ (W^{1,p(x)}_{0,a(x)}(\Omega))^{\ast}=
\overline{\mbox{span}}\left\{e_{n}^{\ast}, \  n\geq 1 \right\}.$$
For any integer $k\geq1$, denote
$$E_{k}=\mbox{span}\left\{e_{k}\right\}, \ \ Y_{k}= \displaystyle \oplus_{j=1}^{k}E_{j}
 \ \ \mbox{and} \ \ Z_{k}=\overline{\displaystyle \oplus_{j=k}^{\infty}E_{j}}.$$

 The main result of this section is the following multiplicity property.

\begin{thm} \label{app}
Assume that $p^{-}>1+s$ and that conditions $(A)$, $(B)$ and $(P)$ are fulfilled.  Then problem
\eqref{imain} has infinitely many solutions.
\end{thm}
\begin{remark}  The main problem in treating equation
\eqref{imain} is the presence of the indefinite potential $a(x)$, which can
vanish at $x_{0}$. To overcome this difficulty, we have proved a new
type of the Caffarelli-Kohn-Nirenberg inequality, Theorem \ref{caf},
which is very useful to prepare the variational framework of
equation \eqref{imain}, for example Lemma \eqref{com}. Moreover, we
 remark that the functions $A$, $B$, $C$, and $D$ that appear in
equation \eqref{imain} are strongly related to our
Caffarelli-Kohn-Nirenberg type theorem. To the best of our knowledge, there are no
known results on the existence of solutions to problem \eqref{imain}. Hence, in order
to prove Theorem \ref{app}, we use the previous section in relationship
with some technical lemma related to the critical point theorem
established by Zou.
\end{remark}

In order to prove Theorem \ref{app} we define the functional $I:
W^{1,p(x)}_{0,a(x)}(\Omega) \rightarrow \mathbb{R}$ by
\begin{align*}
I(u)&=\displaystyle \int_{\Omega} \frac{B(x)}{p(x)} |\nabla
u(x)|^{p(x)}dx+ \displaystyle \int_{\Omega}
\frac{A(x)}{p(x)}|u(x)|^{p(x)}dx+\displaystyle \int_{\Omega}
\frac{C(x)}{p(x)-1}|u(x)|^{p(x)-1}dx\\
&+ \displaystyle \int_{\Omega} \frac{D(x)}{p(x)+1}|u(x)|^{p(x)+1}dx
-\displaystyle \int_{\Omega}b(x)\frac{ |u(x)|^{q(x)}}{q(x)}dx.
\end{align*}
Standard arguments show that $I\in
C^{1}(W^{1,p(x)}_{0,a(x)}(\Omega), \mathbb{R})$ and
\begin{align*}
\langle I'(u),v\rangle&= \displaystyle \int_{\Omega} B(x) |\nabla
u(x)|^{p(x)-2}\nabla u(x) \nabla v(x)dx + \displaystyle
\int_{\Omega}A(x)|u(x)|^{p(x)-2}u(x) v(x)dx \\&+ \displaystyle
\int_{\Omega}D(x)|u(x)|^{p(x)-1}u(x) v(x)dx+ \displaystyle
\int_{\Omega}C(x)|u(x)|^{p(x)-3}u(x)v(x)dx\\& - \int_{\Omega}
b(x)|u(x)|^{q(x)-2}u(x) v(x),
\end{align*}
for all $u,v \in W^{1,p(x)}_{0,a(x)}(\Omega)$. Thus, in order to
find weak solutions of problem \eqref{imain} it suffices to find
critical points of the associated energy  $I$.

 Consider the functional
$$I_{\lambda}(u)=J(u)-\lambda K(u),$$
where \begin{align*} J(u)&=\displaystyle \int_{\Omega}
\frac{B(x)}{p(x)} |\nabla u(x)|^{p(x)}dx+ \displaystyle
\int_{\Omega} \frac{C(x)}{p(x)-1}|u(x)|^{p(x)-1}dx\\&+\displaystyle
\int_{\Omega} \frac{A(x)}{p(x)}|u(x)|^{p(x)}dx+\displaystyle
\int_{\Omega} \frac{D(x)}{p(x)+1}|u(x)|^{p(x)+1}dx
\end{align*}
 and
$$K(u)=\displaystyle \int_{\Omega}b(x)\frac{ |u(x)|^{q(x)}}{q(x)}dx.
$$
Then any critical point of $I_{1}$ is a weak solution of problem
\eqref{imain}.

An important ingredient of the proof of Theorem \ref{app} is the
following version of the fountain theorem, see Zou
\cite{Z}.

\begin{thm}\label{zou}
 Suppose that the functional $I_{\lambda}$ defined above satisfies the following conditions:\\
 $\left(T_{1}\right) \ \ I_{\lambda}$ maps bounded sets to bounded sets uniformly for $\lambda \in \left[1,2\right]$.
 Furthermore, $I_{\lambda}(-u)=I_{\lambda}(u)$ for all $\left(\lambda, u\right)\in \left[1,2\right]\times E$, where $E:=W^{1,p(x)}_{0,a(x)}(\Omega)$;\\
 $\left(T_{2}\right) \ \ B(u)\geq 0, \ \ B(u)\rightarrow \infty$ as $\left\|u\right\|\rightarrow \infty$ on any finite-dimensional subspace of $E$; and\\
 $\left(T_{3}\right)$ there exist $\rho_{k}>r_{k}>0$ such that
 $$a_{k}(\lambda):=\displaystyle \inf_{u\in Z_{k},\left\|u\right\|
 =\rho_{k}}I_{\lambda}(u)\geq 0 >b_{k}(\lambda)=\displaystyle \max_{u\in Y_{k}, \left\|u\right\|=r_{k}}I_{\lambda}(u)\quad\mbox{for}\ \lambda \in \left[1,2\right],$$ $$d_{k}(\lambda)
 =\displaystyle \inf_{u\in Z_{k},\left\|u\right\|\leq \rho_{k}}I_{\lambda}(u)\rightarrow 0\quad\mbox{as $k \rightarrow \infty$ uniformly for $\lambda \in \left[1,2\right]$}.$$

Then there exist a sequence of real numbers $(\lambda_{n})$
converging to 1 and $u(\lambda_{n})\in Y_{n}$ such that
$I_{\lambda_{n}}'|Y_{n}\left(u_{\lambda_{n}}\right)=0$ and
$\left(I_{\lambda_{n}}\right)(u\left(\lambda_{n}\right))\rightarrow
c_{k}\in \left[d_{k}(2), b_{k}(1)\right]$, as $n \rightarrow \infty$.
In particular, for fixed $k\in \mathbb{N}$, if
$\left(u(\lambda_{n})\right)$ has a convergent subsequence
to $u_{k}$, then $I_{1}$ has infinitely many nontrivial
critical points $\left(u_{k}\right)\subset{E}\backslash
\left\{0\right\}$ satisfying $I_{1}\left(u_{k}\right)\rightarrow
0^{-}$ as $k\rightarrow \infty$.
\end{thm}

We start with the following auxiliary property.
\begin{lemma}\label{beta}
Assume that condition $(B)$ holds. Then we have
$$\beta_{k}=\displaystyle \sup_{u\in Z_{k}, \|u\|=1}
 \displaystyle \int_{\Omega}b(x)\frac{ |u(x)|^{q(x)}}{q(x)}dx \rightarrow 0 \ \ \mbox{as} \ \ k\rightarrow +\infty.$$
\end{lemma}
\begin{proof}
It is clear that (we keep "$0<$" since $b>0$ and
$\|u\|=1$) $0<\beta_{k+1}\leq\beta_{k}$, so that
$\beta_{k}\rightarrow \beta\geq 0$ as $k\rightarrow +\infty$. For
every $k\geq 0$, by definition of $\beta_{k}$, there exists
$u_{k}\in Z_{k}$ such that $\left\|u_{k}\right\|=1$ and
$\displaystyle \int_{\Omega}b(x)
\frac{\left|u_{k}\right|^{q(x)}}{q(x)}dx>\frac{\beta_{k}}{2}$. Since
$u_{k}\in Z_{k}$, it follows that $u_{k}\rightharpoonup 0$ in
$W^{1,p(x)}_{0,a(x)}(\Omega)$. Lemma \ref{com} implies that, up to a subsequence,
$$\displaystyle \int_{\Omega}b(x)
\frac{\left|u_{k}\right|^{q(x)}}{q(x)}dx \rightarrow 0\quad \mbox{as}\quad
k\rightarrow +\infty.$$ Thus, $\beta=0$ and the proof is complete.
\end{proof}

The next result establishes that $B$ is coercive on finite-dimensional subspaces of
$W^{1,p(x)}_{0,a(x)}(\Omega)$.

\begin{lemma}\label{glem1}
Assume that hypotheses of Theorem \ref{app} are fulfilled. Then
$K(u)\rightarrow +\infty$ as $\|u\|\rightarrow +\infty$
on any finite-dimensional subspace of $W^{1,p(x)}_{0,a(x)}(\Omega).$
\end{lemma}

\begin{proof}
Let $F$ be a finite-dimensional subspace of
$W^{1,p(x)}_{0,a(x)}(\Omega)$. Put
$$\widetilde{b}(x)=\frac{b(x)}{q(x)}, \ \ \mbox{for all}\ x\in \Omega.$$ We start by showing that there exists $\epsilon_{1}>0$
such that
\begin{equation}\label{epsilon}
m\left\{x\in \Omega; \ \
\widetilde{b}(x)\left|u\right|^{q(x)}\geq
\epsilon_{1}\left\|u\right\|^{q(x)}\right\}\geq \epsilon_{1}, \ \
\mbox{for all}\ u\in{F \backslash  \left\{0\right\}}.
\end{equation}
Otherwise, for any positive integer $n$, there exists $u_{n}\in{
F\backslash  \left\{0\right\}}$ such that
\begin{equation}\label{nnn}
m\left\{x\in \Omega; \ \
\widetilde{b}(x)\left|u_{n}\right|^{q(x)}\geq
\frac{1}{n}\left\|u_{n}\right\|^{q(x)}\right\}<\frac{1}{n}.
\end{equation}
Set $v_{n}(x)=\frac{u_{n}(x)}{\left\|u_{n}\right\|}\in F\backslash
\left\{0\right\}$. Then $\left\|v_{n}\right\|=1$ for all
$n\in\mathbb{N}$ and
$$m\left\{x\in \Omega; \ \ \widetilde{b}(x)\left|v_{n}\right|^{q(x)}\geq \frac{1}{n}\right\}<\frac{1}{n}.$$
Passing to a subsequence, we may assume that $v_{n}\rightarrow
v_{0}$ in $W^{1,p(x)}_{0,a(x)}(\Omega)$ for some $v_{0}\in F$. Then
$\left\|v_{0}\right\|=1$ and, by
 Lemma \ref{com},
\begin{equation}\label{b}
\displaystyle \int_{\Omega}\widetilde{b}(x)
\left|v_{n}-v_{0}\right|^{q(x)}dx \rightarrow 0 \ \ \mbox{as} \ \
n\rightarrow +\infty.
\end{equation}

We claim that there exists $\gamma_{0}>0$ such that
\begin{equation}\label{gamma}
m\left\{x\in \Omega; \ \
\widetilde{b}(x)\left|v_{0}\right|^{q(x)}\geq
\gamma_{0}\right\}\geq \gamma_{0}.
\end{equation}
Indeed, arguing by contradiction, we have
$$m\left\{x\in \Omega; \ \ \widetilde{b}(x)\left|v_{0}\right|^{q(x)}\geq \frac{1}{n}\right\}=0, \ \ \mbox{for all}\ n\in{\mathbb{N}}.$$
It follows that
$$0\leq \displaystyle \int_{\Omega}\widetilde{b}(x) \left|v_{0}\right|^{q(x)+1}dx<\frac{\left\|v_{0}\right\|_{1}}{n}
\rightarrow 0, \ \ \mbox{as} \ \ n \rightarrow +\infty.$$ Hence
$v_{0}=0$, which contradicts $\left\|v_{0}\right\|=1.$

Set
$$\Omega_{0}=\left\{x\in \Omega; \ \ \widetilde{b}(x)\left|v_{0}\right|^{q(x)}\geq \gamma_{0}\right\},
\ \    \Omega_{n}=\left\{x\in \Omega;
 \ \ \widetilde{b}(x)\left|v_{n}\right|^{q(x)}< \frac{1}{n}\right\}$$
 and  $$\Omega_{n}^{c}=\left\{x\in
 \Omega;
  \ \ \widetilde{b}(x)\left|v_{n}\right|^{q(x)}\geq \frac{1}{n}\right\}.$$
By \eqref{nnn} and \eqref{gamma}, we obtain
\begin{align*}
m\left(\Omega_{n}\cap \Omega_{0}\right)&=m\left(\Omega_{0}\backslash \left(\Omega_{n}^{c}\cap \Omega_{0}\right)\right)\\
&\geq m\left( \Omega_{0}\right)-m\left(\Omega_{n}^{c}\cap \Omega_{0}\right)\\
&\geq \gamma_{0}-\frac{1}{n}>\frac{\gamma_{0}}{2}
\end{align*}
for large enough $n$. Consequently, (the inequality is
correct since $q(x)>1$, $y^{q(x)}$ is convex and we write
$|v_{0}|=|v_{0}-v_{n}+v_{n}|$)
\begin{align*}
\displaystyle
\int_{\Omega}\widetilde{b}(x)\left|v_{n}-v_{0}\right|^{q(x)}dx
&\geq
 \displaystyle \int_{\Omega_{n}\cap \Omega_{0}}\widetilde{b}(x)\left|v_{n}-v_{0}\right|^{q(x)}dx\\
&\geq \frac{1}{2^{q^{+}-1}}\int_{\Omega_{n}\cap
\Omega_{0}}\widetilde{b}(x)\left|v_{0}\right|^{q(x)}dx
-\int_{\Omega_{n}\cap \Omega_{0}}\widetilde{b}(x)\left|v_{n}\right|^{q(x)}dx\\
&\geq \left(\frac{\gamma_{0}}{2^{q^{+}-1}}-\frac{1}{n}\right) m\left(\Omega_{n}\cap \Omega_{0}\right)\\
&\geq\frac{\gamma_{0}^{2}}{2^{q^{+}+1}}>0,
\end{align*}
for all large $n$, which is a contradiction to \eqref{b}. Therefore
\eqref{epsilon} holds. For the $\epsilon_{1}$ given in
\eqref{epsilon}, let
$$\Omega_{u}=\left\{x\in \Omega; \ \ \widetilde{b}(x)\left|u\right|^{q(x)}\geq
 \epsilon_{1}\left\|u\right\|^{q(x)}\right\}, \ \ \mbox{for all}\ u\in F\backslash \left\{0\right\}. $$
Then
\begin{equation}\label{omegau}
m\left(\Omega_{u}\right)\geq \epsilon_{1} \ \ \mbox{for all} \ \
u\in F\backslash \left\{0\right\}.
\end{equation}
From $\left(B\right)$ and \eqref{omegau}, for any $u\in F\backslash
\left\{0\right\}$  with $\|u\|\geq 1$, we get
\begin{align*}
K(u)&=\displaystyle
\int_{\Omega}\widetilde{b}(x)\left|u\right|^{q(x)}dx
 \geq \displaystyle \int_{\Omega_{u}}\widetilde{b}(x)\left|u\right|^{q(x)}dx\\
&\geq \epsilon_{1} \left\|u\right\|^{q^{-}}
m\left(\Omega_{u}\right)\geq
\epsilon_{1}^{2}\left\|u\right\|^{q^{-}}.
\end{align*}
This implies that $K(u)\rightarrow \infty$ as
$\left\|u\right\|\rightarrow \infty$ on any finite-dimensional
subspace of $E$ and this completes the proof.
\end{proof}

\begin{lemma}\label{glem2}
Suppose that the conditions of Theorem \ref{app} are satisfied. Then
there exists a sequence $\rho_{k}\rightarrow 0^{+}$ as $k\rightarrow
+\infty$ such that
$$a_{k}(\lambda)=\displaystyle \inf_{u\in Z_{k}, \|u\|=\rho_{k}}I_{\lambda}(u)\geq 0, \ \ \mbox{for all}\ k\geq k_{1}$$
and
$$d_{k}(\lambda)=\displaystyle \inf_{u\in Z_{k}, \|u\|\leq \rho_{k}}I_{\lambda}(u) \rightarrow 0 \ \
\mbox{as} \ \ k\rightarrow +\infty \ \ \mbox{uniformly for all} \ \
\lambda \in [1,2]. $$
\end{lemma}

\begin{proof}
By Proposition \ref{pr1} and Lemma \ref{com}, we deduce that for any $u\in
Z_{k}$ with $\|u\|<1$, we have
\begin{equation}\label{n}
\begin{array}{ll}
I_{\lambda}(u)&\displaystyle\geq  \frac{1}{p^{+}}\left(\displaystyle \int_{\Omega}
B(x) |\nabla u(x)|^{p(x)}dx+\displaystyle \int_{\Omega}
A(x)|u(x)|^{p(x)}dx\right)\\
&\displaystyle + \frac{1}{p^{+}+1}\displaystyle \int_{\Omega}
D(x) | u(x)|^{p(x)+1}dx\\
&\displaystyle +\frac{1}{p^{+}-1} \displaystyle
\int_{\Omega}C(x)|u(x)|^{p(x)-1}dx- \lambda\displaystyle
\int_{\Omega}\frac{b(x)}{q(x)} |u(x)|^{q(x)}dx
\\ &\displaystyle\geq \frac{1}{4^{p^{+}+2}(p^{+}+1)}
\|u\|^{p^{+}+1}-\lambda\|u\|^{q^{-}}\displaystyle
\int_{\Omega}\frac{b(x)}{q(x)}(\frac{|u(x)|}{\|u\|})^{q(x)}dx \\
&\displaystyle\geq \frac{1}{4^{p^{+}+2}(p^{+}+1)}
\|u\|^{p^{+}+1}-\frac{2\beta_{k}}{q^{-}}\|u\|^{q^{-}}.
\end{array}
\end{equation}

 We denote
$\rho_{k}=(\frac{4^{p^{+}+3}(p^{+}+1)\beta_{k}}{q^{-}})^{\frac{1}{p^{+}+1-q^{-}}}$.
By invoking Lemma \ref{beta} we can deduce that $\rho_{k}\rightarrow 0$ as
$k\rightarrow +\infty$. Then there exists $k_{1}\in \mathbb{N}$ such
that $\rho_{k}\leq  \frac{1}{4^{p^{+}+3}(p^{+}+1)}$ for all $k\geq
k_{1}$. Relation \eqref{n} implies that
$$a_{k}(\lambda)=\displaystyle \inf_{u\in Z_{k},\|u\|=\rho_{k}}I_{\lambda}(u)
\geq \frac{1}{2.4^{p^{+}+3}(p^{+}+1)}\rho_{k}^{p^{+}+1}, \ \
\mbox{for all} \ \ k\geq k_{1}.$$ Furthermore, by \eqref{n}, we have
$$\displaystyle \inf_{u\in Z_{k}, \|u\|\leq
\rho_{k}}I_{\lambda}(u)\geq -\frac{2\beta_{k}}{q^{-}}\|u\|^{q^{-}},
\ \ \mbox{for all}\ k\geq k_{1}.$$ Having in mind
$I_{\lambda}(0)=0$, then
$$\inf_{u\in Z_{k}, \|u\|\leq
\rho_{k}}I_{\lambda}(u)\leq0, \ \ \forall k\geq k_{1}.$$
 Using the fact that $\beta_{k},\rho_{k}\rightarrow 0$ as
$k\rightarrow +\infty$ and the above inequalities, we
deduce that
$$d_{k}(\lambda)=\displaystyle \inf_{u\in Z_{k}, \|u\|=\rho_{k}}I_{\lambda}(u) \rightarrow 0
\ \ \mbox{as} \ \ k\rightarrow +\infty \ \ \mbox{uniformly for all} \ \
\lambda \in [1.2].$$ This completes the proof.
\end{proof}

\begin{lemma}\label{glem3}
Assume that hypotheses of Theorem \ref{app} are fulfilled. Then, for
the sequence obtained in Lemma \ref{glem2}, there exists
$0<r_{k}<\rho_{k}$ for all $k\in \mathbb{N}$ such that
$$b_{k}(\lambda)=\displaystyle \max_{u\in Y_{k}, \|u\|=r_{k}}I_{\lambda}(u)<0 \ \ \mbox{for all} \ \ \lambda \in [1,2].$$
\end{lemma}
\begin{proof}
Let $u\in Y_{k}$ with $\|u\|<1$ and $\lambda \in [1,2]$. By $(A)$,
$(P)$ and \eqref{epsilon}, there exists $\epsilon_{k}>0$ such that
\begin{align*}
I_{\lambda}(u)&=\displaystyle \int_{\Omega} \frac{B(x)}{p(x)}
|\nabla u(x)|^{p(x)}dx+ \displaystyle \int_{\Omega}
\frac{A(x)}{p(x)}|u(x)|^{p(x)}dx+\displaystyle
\int_{\Omega}\frac{C(x)}{p(x)-1}|u|^{p(x)-1}dx\\&+\displaystyle
\int_{\Omega} \frac{D(x)}{p(x)+1}|u(x)|^{p(x)+1}dx
-\lambda \displaystyle \int_{\Omega}b(x)\frac{ |u(x)|^{q(x)}}{q(x)}dx\\
&\leq(\frac{2}{p^{-}}+
\frac{1}{p^{-}+1}+\frac{1}{p^{-}-1})\|u\|^{p^{-}-1}-\epsilon_{k}\|u\|^{q^{-}}
m(\Omega_{u})\\
&\leq (\frac{2}{p^{-}}+
\frac{1}{p^{-}+1}+\frac{1}{p^{-}-1})\|u\|^{p^{-}-1}-\epsilon_{k}^{2}\|u\|^{q^{-}}.
\end{align*}
Since $0<q^{-}<q^{+}< p^{-}<p^{+}$, we deduce that for small
$\|u\|=r_{k}$ we have
$$b_{k}(\lambda)<0, \ \ \mbox{for all}\ k\in \mathbb{N}.$$
 This concludes the proof of Lemma \ref{glem3}.
 \end{proof}

\subsection{Proof of Theorem \ref{app} completed.} Evidently, condition
$\left(T_{1}\right)$ in Theorem \ref{zou} holds. By Lemmas
\ref{glem1}, \ref{glem2} and \ref{glem3}, conditions
$\left(T_{2}\right)$ and $\left(T_{3}\right)$ in Theorem \ref{zou}
are satisfied. Then, by Theorem \ref{zou} there exist
$\lambda_{n}\rightarrow 1$ and $u(\lambda_{n})\in Y_{n}$ such that
$$I'_{\lambda_{n}}|Y_{n}(u(\lambda_{n}))=0, \ \
 I_{\lambda_{n}}(u(\lambda_{n}))\rightarrow c_{k}\in \left[d_{k}(2), b_{k}(1)\right]$$
as $n\rightarrow +\infty.$

For the sake of notational simplicity, we always set in what follows
$u_{n}=u\left(\lambda_{n}\right)$ for all $n\in \mathbb{N}.$

\smallskip
\textsc{Claim:} the sequence $(u_{n})$ is bounded in
$W^{1,p(x)}_{0,a(x)}(\Omega).$

Arguing by contradiction, we suppose that  $(u_{n})$ is
unbounded in $W^{1,p(x)}_{0,a(x)}(\Omega)$. Without loss of
generality, we can assume that $\|u_{n}\|> 1$ for all $n\geq1$.

Observe first that there exists $c>0$ such that for large enough
$n$,
\begin{equation}\label{ba}
\langle I'_{\lambda_{n}}(u_{n}), u_{n}\rangle \leq \|u_{n}\| \ \
\mbox{and} \ \ |I_{\lambda_{n}}(u_{n})|\leq c.
\end{equation}
Using relation \eqref{ba}, we have
\begin{equation}\label{new}
\begin{array}{ll}
c&\displaystyle\geq I_{\lambda_{n}}(u_{n})\geq\frac{1}{p^{+}}(\displaystyle
\int_{\Omega} B(x) |\nabla u_{n}(x)|^{p(x)}dx+ \displaystyle
\int_{\Omega}
A(x)|u_{n}(x)|^{p(x)}dx)\\
&\displaystyle +\frac{1}{p^{+}-1}\displaystyle
\int_{\Omega}C(x)|u_{n}|^{p(x)-1}dx\\
&\displaystyle +\frac{1}{p^{+}+1}\displaystyle \int_{\Omega}
D(x)|u_{n}(x)|^{p(x)+1}dx -\frac{2}{q^{-}} \displaystyle
\int_{\Omega}b(x) |u_{n}(x)|^{q(x)}dx.
\end{array}
\end{equation}
Combining Proposition \ref{pr1}, relation \eqref{new} and since
$q^{+}<p^{-}-1<p^{-}<p^{+}<p^{+}+1$, it follows that $(u_{n})$ is
bounded in $W^{1,p(x)}_{0,a(x)}(\Omega)$. This shows that our claim
is true. So, by Lemma \ref{com} and up to a
subsequence, we can assume that
$$u_{n}\rightharpoonup u_{0} \ \ \mbox{in} \ \ W^{1,p(x)}_{0,a(x)}(\Omega)$$
and
$$u_{n}\rightarrow u_{0} \ \ \mbox{in} \ \ L^{q(x)}(\Omega).$$

In what follows, we show that
$$u_{n}\rightarrow u_{0} \ \ \mbox{in} \ \ W^{1,p(x)}_{0,a(x)}(\Omega).$$
Having in mind that $(u_{n})$ is a bounded sequence, we get
\begin{equation}\label{p1}
\displaystyle \lim_{n \rightarrow +\infty}\langle
I_{\lambda_{n}}'(u_{n})-I'_{\lambda_{n}}(u_{0}),u_{n}-u_{0}\rangle=0.
\end{equation}
Hence, \eqref{p1} and Lemma \ref{com} give as $n\rightarrow +\infty$
\begin{align*}
o(1)&=\langle I_{\lambda_{n}}'(u_{n})-I'_{\lambda_{n}}(u_{0}),u_{n}-u_{0}\rangle\\
&=\displaystyle \int_{\Omega} B(x)(|\nabla u_{n}(x)|^{p(x)-2}\nabla
u_{n}(x)-|\nabla
u_{0}(x)|^{p(x)-2}\nabla u_{0}(x))(\nabla u_{n}(x)-\nabla u_{0}(x))dx\\
&+\displaystyle \int_{\Omega}
A(x)(|u_{n}(x)|^{p(x)-2}u_{n}(x)-|u_{0}(x)|^{p(x)-2}u_{0}(x))(u_{n}(x)-u_{0}(x))dx\\
&+\displaystyle \int_{\Omega}
D(x)(|u_{n}(x)|^{p(x)-1}u_{n}(x)-|u_{0}(x)|^{p(x)-1}u_{0}(x))(u_{n}(x)-u_{0}(x))dx\\
&+ \displaystyle \int_{\Omega}
C(x)(|u_{n}(x)|^{p(x)-3}u_{n}(x)-|u_{0}(x)|^{p(x)-3}u_{0}(x))(u_{n}(x)-u_{0}(x))dx.
\end{align*}

 We have for all $n\in \mathbb{N}$
$$\int_{\Omega} B(x)(|\nabla u_{n}(x)|^{p(x)-2}\nabla
u_{n}(x)-|\nabla u_{0}(x)|^{p(x)-2}\nabla u_{0}(x))(\nabla
u_{n}(x)-\nabla u_{0}(x))dx\geq 0,$$
$$\displaystyle \int_{\Omega} A(x)(|u_{n}(x)|^{p(x)-2}u_{n}(x)-|u_{0}(x)|^{p(x)-2}u_{0}(x))(u_{n}(x)-u_{0}(x))dx\geq
0,$$
$$\displaystyle \int_{\Omega}C(x)(|u_{n}(x)|^{p(x)-3}u_{n}(x)-|u_{0}(x)|^{p(x)-3}u_{0}(x))(u_{n}(x)-u_{0}(x))dx\geq
0,$$
 and
$$ \displaystyle \int_{\Omega}
D(x)(|u_{n}(x)|^{p(x)-1}u_{n}(x)-|u_{0}(x)|^{p(x)-1}u_{0}(x))(u_{n}(x)-u_{0}(x))dx\geq
0.$$
Therefore
\begin{equation}\label{p2}
\displaystyle \lim_{n\rightarrow +\infty}\int_{\Omega} B(x)(|\nabla
u_{n}(x)|^{p(x)-2}\nabla u_{n}(x)-|\nabla u_{0}(x)|^{p(x)-2}\nabla
u_{0}(x))(\nabla u_{n}(x)-\nabla u_{0}(x))dx= 0,
\end{equation}
 \begin{equation}\label{p3}
\displaystyle \lim_{n\rightarrow +\infty}\displaystyle \int_{\Omega}
A(x)
(|u_{n}(x)|^{p(x)-2}u_{n}(x)-|u_{0}(x)|^{p(x)-2}u(x))(u_{n}(x)-u_{0}(x))dx=0,
 \end{equation}
 \begin{equation}\label{p0}
\displaystyle \lim_{n\rightarrow +\infty}\displaystyle \int_{\Omega}
C(x)
(|u_{n}(x)|^{p(x)-3}u_{n}(x)-|u_{0}(x)|^{p(x)-3}u(x))(u_{n}(x)-u_{0}(x))dx=0,
 \end{equation}
 and
 \begin{equation}\label{p4}
\displaystyle \lim_{n\rightarrow +\infty}\displaystyle \int_{\Omega}
D(x)(|u_{n}(x)|^{p(x)-1}u_{n}(x)-|u_{0}(x)|^{p(x)-1}u_{0}(x))(u_{n}(x)-u_{0}(x))dx=0.
 \end{equation}
 Let us now recall the Simon inequalities \cite[formula 2.2]{simon} (see also \cite[p. 713]{fprcpde})
\begin{equation}\label{sim}
\begin{cases}
\left|x-y\right|^{p}\leq
c_{p}\left(\left|x\right|^{p-2}x-\left|y\right|^{p-2}y\right).(x-y)
\; \; \; \; \;  \; \; \; \; \;   \; \; \; \; \; \; \; \; \; \; \; \; \; \; \; \; \; \mbox{for} \ \ p\geq 2 \\
 \left|x-y\right|^{p} \leq C_{p}\left[\left(\left|x\right|^{p-2}x-\left|y\right|^{p-2}y\right)
 .(x-y)\right]^{\frac{p}{2}}\left(\left|x\right|^{p}+\left|y\right|^{p}\right)^{\frac{2-p}{2}} \; \; \mbox{for} \ \ 1<p< 2,
\end{cases}
\end{equation}
for all $x,y\in \mathbb{R}^{N}$, where $c_{p}$ and $C_{p}$ are
positive constants depending only on $p$. Combining \eqref{p2},
\eqref{p3}, \eqref{p4}, \eqref{p4} and \eqref{sim}, we conclude that
$$\displaystyle \lim_{n\rightarrow +\infty}\|u_{n}-u_{0}\|=0.$$
 Now, invoking  Theorem
\ref{zou}, we complete the proof of  Theorem \ref{app}.\qed

\begin{remark}  We point out that the multiplicity property described in Theorem \ref{app} is somehow related with Theorem 1.1 established in Bahrouni \cite{bahrouni}. However, there are several differences between problem \eqref{imain} studied in this paper and problem (1.1) considered in \cite{bahrouni}.
 For instance, the main result in \cite{bahrouni} is concerned with the existence of infinitely many solutions (as in our case) but for a class of {\it semilinear elliptic equations} driven by the Laplace equation and with a reaction term defined by the sum of two power-type {\it concave} terms.
Problem \eqref{imain} in the present work has a much more complicated structure. For instance, the non-homogeneous differential operator is perturbed by two power-type terms with variable exponent. Moreover, in the present work
 we are concerned with competition effects between several {\it variable exponents} and indefinite potentials. A crucial role in the analysis developed in the present paper is played by the main abstract result established in the first part of this paper, namely the weighted version of the Caffarelli-Kohn-Nirenberg inequality for variable exponents. Such an abstract result (even for constant exponents) is not used in \cite{bahrouni}. The  analysis carried out in this paper includes the {\it degenerate} case, which corresponds to a potential that can vanish in one or more points. Finally, it is worth pointing out that this potential is assumed to be {\it indefinite} and not positive, as in \cite{bahrouni}.
\end{remark}  

\subsection{Perspectives and open problems}
The methods developed in this paper can be extended to more general variational integrals. We mainly refer to energy functionals associated
to non-homogeneous operators of the type $-{\rm div}\, (\phi(x,|\nabla u|)\nabla u)$, which extend the standard $p(x)$-Laplace operator.
These operators have been introduced by Kim and Kim \cite{kim}; see also Baraket, Chebbi, Chorfi, and R\u adulescu \cite{baraket} for recent advances in this new abstract setting.

We believe that a valuable research direction is to generalize the abstract approach developed in this paper to the framework of
 {\it double-phase} variational integrals studied by Mingione {\it et al.} \cite{mingi1, mingi2}. We expect that a related Caffarelli-Kohn-Nirenberg inequality can be established for  energies of the type
\begin{equation}\label{integr1}u\mapsto \int_\Omega \left[|\nabla u|^{p(x)}+a(x)|\nabla u|^{q(x)}\right]dx\end{equation}
or
\begin{equation}\label{integr2}u\mapsto \int_\Omega \left[|\nabla u|^{p(x)}+a(x)|\nabla u|^{q(x)}\log(e+|x|)\right]dx,\end{equation}
where $p(x)\leq q(x)$, $p\not=q$, and $a(x)\geq 0$. In the case of two different materials that involve power hardening exponents $p(x)$
and $q(x)$, the coefficient $a(x)$ describes the geometry of a composite of these two materials. When $a(\cdot)>0$ then the $q(\cdot)$-material is present.
In the opposite case,  the $p(\cdot)$-material is the only one describing the composite. We also point out that since the integral energy functional defined in
 \eqref{integr2} has a degenerate behavior on the zero set of the gradient, it is natural to study what happens if the integrand is modified in such a way that,
 also if $|\nabla u|$ is small, there exists an imbalance between the two terms of the integrand.

\medskip
{\bf Acknowledgements.} The authors thank both anonymous referees for their careful reading of our paper and for their remarks and comments, which have considerably improved the initial version of this work.
V.R. and D.R. were partially supported by the Slovenian Research Agency grants P1-0292, J1-8131, J1-7025, and N1-0064.
 V.R. also acknowledges the support through a grant of the Romanian Ministry of Research and Innovation, CNCS-UEFISCDI, project number PN-III-P4-ID-PCE-2016-0130, within PNCDI III.

\end{document}